\documentclass[11pt]{article}
 
\usepackage{epsfig}
\usepackage{latexsym}
\usepackage{amsfonts}
\usepackage{amssymb}

\newtheorem{definition}{Definition}[section]
\newtheorem{proposition}[definition]{Proposition}

\newtheorem{theorem}[definition]{Theorem}
\newtheorem{corollary}[definition]{Corollary}
\newtheorem{conjecture}[definition]{Conjecture}

\begin{document}

\baselineskip=7mm

\begin{center}

\vspace*{5mm}

{\Large \bf Completely Independent Spanning Trees in Line Graphs} 

\vspace*{25mm}

{\large Toru Hasunuma}

\vspace*{10mm}

{Department of Mathematical Sciences, \\
Tokushima University, \\
2--1 Minamijosanjima, Tokushima 770--8506 Japan \\
{\tt hasunuma@tokushima-u.ac.jp}
}

\end{center}

\vspace*{10mm}

\baselineskip=5.5mm

\noindent\textbf{Abstract:}
Completely independent spanning trees in a graph $G$ are spanning trees of $G$ such that
for any two distinct vertices of $G$, the paths between them in the spanning trees are 
pairwise edge-disjoint and internally vertex-disjoint.
In this paper, we present a tight lower bound on the maximum number of completely independent 
spanning trees in $L(G)$, where $L(G)$ denotes the line graph of a graph $G$. 
Based on a new characterization of a graph with $k$ completely independent spanning trees,
we also show that for any complete graph $K_n$ of order $n \geq 4$, there are 
$\lfloor \frac{n+1}{2} \rfloor$ completely independent spanning trees in $L(K_n)$
where the number $\lfloor \frac{n+1}{2} \rfloor$ is optimal, 
such that $\lfloor \frac{n+1}{2} \rfloor$ completely independent spanning trees still exist
in the graph obtained from $L(K_n)$ by deleting any vertex (respectively, any induced path 
of order at most $\frac{n}{2}$) for $n = 4$ or odd $n \geq 5$ (respectively, even $n \geq 6$). 
Concerning the connectivity and the number of completely independent spanning trees,
we moreover show the following, where $\delta(G)$ denotes the minimum degree of $G$. 
\begin{itemize}
\item Every $2k$-connected line graph $L(G)$ has $k$ completely independent spanning trees 
if $G$ is not super edge-connected or $\delta(G) \geq 2k$.
\item Every $(4k-2)$-connected line graph $L(G)$ has $k$ completely independent spanning trees 
if $G$ is regular.
\item Every $(k^2+2k-1)$-connected line graph $L(G)$ with $\delta(G) \geq k+1$ has $k$ completely independent spanning trees. 
\end{itemize}

\vspace*{5mm}


\newpage

\bigskip

\baselineskip=7mm

\section{Introduction}

For two sets $A$ and $B$,
$A \setminus B$ denotes the set difference $\{x\ |\ x \in A, x \not\in B \}$.
Throughout the paper, a graph $G=(V,E)$ means a simple undirected graph.
The degree of a vertex $v$ in $G$ is denoted by ${\rm deg}_G(v)$.
Let $\delta(G) = \min_{v \in V(G)}{\rm deg}_G(v)$.
We denote by $\kappa(G)$ and $\lambda(G)$ the connectivity and the edge-connectivity of $G$,
respectively. 
For every graph $G$, it holds that 
$\kappa(G) \leq \lambda(G) \leq \delta(G)$. 
A graph $G$ is $k$-connected (respectively, $k$-edge-connected) 
if $\kappa(G) \geq k$ (respectively, $\lambda(G) \geq k$).
For $S \subsetneq V(G)$ or $S \subseteq E(G)$, 
$G-S$ denotes the graph obtained from $G$ by deleting every element in $S$. 
When $S$ consists of only one element $s$, $G-S$ may be abbreviated to $G-s$.
For a nonempty subset $S$ of $V(G)$, the subgraph of $G$ induced by $S$ is denoted by 
$\langle S \rangle_G$, i.e., $\langle S \rangle_G = G - (V(G) \setminus S)$.
When $\langle S \rangle_G$ is a path, $\langle S \rangle_G$ is called an induced path of $G$. 
An edge-cut of a connected graph $G$ is a subset $F \subseteq E(G)$ such that $G-F$ is disconnected.
A connected graph $G$ is called {\em super edge-connected} if every minimum edge-cut
isolates a vertex.
The line graph $L(G)$ of $G$ is the graph with $V(L(G))=E(G)$ in which two vertices are adjacent
if and only if their corresponding edges are adjacent, i.e., they are incident to 
a common vertex in $G$.

{\em Completely independent spanning trees} in a graph $G$
are spanning trees of $G$ such that for any two distinct vertices $u, v$ of $G$,
the paths between $u$ and $v$ in the spanning trees  
mutually have no common edge and no common vertex except for $u$ and $v$. 
Completely independent spanning trees find applications in fault-tolerant broadcasting 
problems \cite{H01} and protection routings \cite{PCC} in communication networks. 
Motivated by these applications, completely independent spanning trees have been widely
studied for interconnection networks (e.g., see \cite{CCW,CPHYC,HM,PCC2}).
From a theoretical point of view, 
sufficient conditions for Hamiltonian graphs such as Dirac's condition and Ore's condition
were shown to be also sufficient for a graph to have two completely independent spanning trees
\cite{Ar,FHL}; moreover, Dirac's condition were 
generalized to minimum degree conditions for a graph to have $k$ completely independent 
spanning trees \cite{H15,HL}.
From an algorithmic point of view, it has been proved that
the problem of deciding whether a given graph has two completely independent spanning trees 
is NP-complete \cite{H02}.

Let $\tau(G)$ denote the maximum number of edge-disjoint spanning trees in $G$.
Let $S \subsetneq V(G)$.
For $v \in S$, let 
$$f_{S}(v) = {\rm deg}_G(v) -{\rm deg}_{\langle S \rangle_G}(v).$$
Define $\zeta(S)$ as follows: 
$$\zeta(S) = \left\{ \begin{array}{ll}
\min_{uv \in E(\langle S \rangle_G)}f_{S}(u)+f_{S}(v)  & \mbox{ if } 
S \neq \emptyset \mbox{ and } E(\langle S \rangle_G) \neq \emptyset, \\
\infty & \mbox{ otherwise.}
\end{array} \right. $$
A star is a tree which has at most one vertex with degree at least two.
If $S = \emptyset$ or every component of $\langle S \rangle_G$ 
is a star, then $S$ is called a {\it star-subset} of $V(G)$.
Let ${\mathcal F}(V(G))$ denote the family of star-subsets of $V(G)$.
Define $\tau'(G)$ as follows: 
$$\tau'(G) = \max_{S \in {\mathcal F}(V(G))}\min\{\tau(G-S), \zeta(S)\}.$$
Since $\min\{\tau(G-\emptyset),\zeta(\emptyset)\} = \tau(G)$, 
it holds that $\tau'(G) \geq \tau(G)$ for any graph $G$.  
Let $\tau^\ast(G)$ denote the maximum number of completely independent spanning trees in $G$.
In this paper, we show the following. 

\begin{theorem} \label{main-0}
Every line graph $L(G)$ has $\tau'(G)$ 
completely independent spanning trees, i.e., 
$\tau^\ast(L(G)) \geq \tau'(G)$. 
There exists a graph $G$ with $\tau'(G) > \tau(G)$ such that
$\tau^\ast(L(G)) = \tau'(G)$. 
Moreover, there exists a graph $G$ such that 
$\tau^\ast(L(G)) = \tau(G)$.
\end{theorem}

For any complete graph $K_n$ of order $n \geq 4$, 
Wang et al. \cite{Wang-} recently presented
an algorithm to construct 
$\lfloor \frac{n+1}{2} \rfloor$ completely independent spanning trees in $L(K_n)$, 
where the number $\lfloor \frac{n+1}{2} \rfloor$
of completely independent spanning trees is optimal, i.e., $\tau^\ast(L(K_n)) = \lfloor \frac{n+1}{2} \rfloor$.
They also implemented their algorithm to verify its validity. 
Based on a new characterization of a graph with $k$ completely independent spanning trees, 
we show the following result stronger than the result of Wang et al. by presenting a direct proof. 

\begin{theorem} \label{main-01} 
For any $n \geq 4$, there are $\lfloor \frac{n+1}{2} \rfloor$ completely independent 
spanning trees in $L(K_n)$ such that 
for $n = 4$ or any odd $n \geq 5$ and any vertex $v$ of $L(K_n)$
(respectively, any even $n \geq 6$ and any induced path $P$ of order at most $\frac{n}{2}$ of $L(K_n)$), 
$\lfloor \frac{n+1}{2} \rfloor$ completely independent spanning trees still exist
in $L(K_n)-v$ (respectively, $L(K_n)-V(P)$).
\end{theorem}

The notion of completely independent spanning trees was introduced in \cite{H01} by
strengthening the notion of independent spanning trees. 
Independent spanning trees have a specific vertex called the root and focus only on  
the paths between the root and any other vertex.
Thus, completely independent spanning trees are independent spanning trees rooted at any given vertex.
For independent spanning trees, the following conjecture is well-known. 

\begin{conjecture}
For any $k \geq 2$, 
every $k$-connected graph has $k$ independent spanning trees rooted at any vertex.
\end{conjecture}

This conjecture has been shown to be true for $k \leq 4$ (e.g., see \cite{CLY,OY})
and for planar graphs \cite{Huck,Huck2}; however, it remains open for general $k \geq 5$.
Unlike independent spanning trees, 
it has already been shown that there is no direct relationship 
between the connectivity and the number of completely independent spanning trees. 
Namely, P\'{e}terfalvi \cite{P12} proved that for any $k \geq 2$, 
there exists a $k$-connected graph which 
has no two completely independent spanning trees. 

By definition, independent spanning trees are not always edge-disjoint, 
while completely independent spanning trees are edge-disjoint. 
Thus, the notion of completely independent spanning trees is also stronger than 
the notion of edge-disjoint spanning trees.
For edge-disjoint spanning trees, the following theorem is well-known.

\begin{theorem} {\rm (Nash-Williams \cite{NW}, Tutte \cite{T})} \label{NWT}
For any $k \geq 2$, 
every $2k$-edge-connected graph has $k$ edge-disjoint spanning trees.
\end{theorem}

From the result of P\'{e}terfalvi, 
the corresponding statement for $k$ completely independent spanning trees 
in a $2k$-connected graph does not hold in general; 
however, it has been shown that such a statement holds 
if we restrict ourselves to specific graphs 
such as 4-connected maximal planar graphs \cite{H02}, 2-dimensional torus networks \cite{HM}
and augmented cubes with dimensions at most five \cite{Mane-}. 
As far as we know, no similar statement for general $k \geq 2$ 
has been shown even for specific graphs. 
Combining our lower bound on $\tau^\ast(L(G))$
with previously known results, 
we show the following results on $L(G)$ concerning the connectivity
and the number of completely independent spanning trees. 

\begin{theorem} \label{main-1}
For any $k \geq 2$, 
every $2k$-connected line graph $L(G)$ has $k$ completely independent spanning trees 
if $G$ is not super edge-connected or $\delta(G) \geq 2k$. 
\end{theorem}

\begin{theorem} \label{main-3}
For any $k \geq 2$, 
every $(k^2+2k-1)$-connected line graph $L(G)$ with $\delta(G) \geq k+1$
has $k$ completely independent spanning trees.  
\end{theorem}

In particular, we have the following corollary by setting $k = 2$ in Theorem \ref{main-1}.

\begin{corollary} \label{cor-main-01}
Every 4-connected line graph $L(G)$ with $\delta(G) \geq 4$
has two completely independent spanning trees.
\end{corollary}

It has been proved in \cite{FHL} that for any cubic graph $G$ of order at least 9,
$L(G)$ has no two completely independent spanning trees. 
There indeed exists a cubic graph $G$ of order at least 9 such that $L(G)$ is 4-connected,
e.g., the Petersen graph; in fact, it has also been pointed out in \cite{FHL} that
there are infinitely many such cubic graphs. 
Thus, there exists a 4-connected line graph $L(G)$ with $\delta(G) = 3$ which 
has no two completely independent spanning trees. 
In this sense, the lower bound of $4$ on $\delta(G)$ in Corollary \ref{cor-main-01}
is best possible. 
On the other hand, by increasing the connectivity of $L(G)$,
the lower bound of $4$ on $\delta(G)$ can be weakened to $3$.
Namely, from Theorem \ref{main-3}, the following corollary is obtained.

\begin{corollary} \label{cor-main-02}
Every 7-connected line graph $L(G)$ with $\delta(G) \geq 3$
has two completely independent spanning trees.
\end{corollary}

Suppose that $L(G)$ is $(4k-2)$-connected and $G$ is $r$-regular, i.e., every vertex in $G$ 
has degree $r$.
Since $L(G)$ is $(2r-2)$-regular, it holds that $\delta(L(G)) = 2r-2 \geq \kappa(L(G)) \geq 4k-2$.
Thus, $r \geq 2k$.
Therefore, from Theorem \ref{main-1}, 
we have the following result on $L(G)$ for a regular graph $G$ without any other restriction on $G$.

\begin{corollary} \label{main-cor-1}
For any $k \geq 2$, 
every $(4k-2)$-connected line graph $L(G)$ has $k$ completely independent spanning trees 
if $G$ is regular. 
\end{corollary}

The maximum number of completely independent spanning trees
in $L(K_n)$ is $\lfloor \frac{n+1}{2} \rfloor$ as shown in \cite{Wang-} and 
the connectivity of $L(K_n)$ is $2n-4$ for $n \geq 4$ 
(these facts can be briefly confirmed in Sections 3 and 5).
Thus, $L(K_{2k+1})$ is $(4k-2)$-connected but has no $k+2$ completely independent spanning trees
for $k \geq 2$. 
Hence, the lower bound of $k$ on $\tau^\ast(L(G))$ in Corollary \ref{main-cor-1} 
cannot be improved to $k+2$ in general. 
In this sense, we may say that the lower bound of $k$ is nearly optimal.

This paper is organized as follows. 
Section 2 presents a new characterization of a graph $G$ 
with $k$ completely independent spanning trees
which can be applied to faulty networks and also from which a general upper bound on $\tau^\ast(G)$
is obtained.
Proofs of Theorems \ref{main-0} and \ref{main-01} are given in Sections 3 and 4, respectively. 
Section 5 presents statements stronger than Theorem \ref{main-1} and a statement with  
a condition complementary to that of Theorem \ref{main-1}, from which Theorem \ref{main-3}
is obtained.

\section{Characterizations}

A  dominating set $S$ of a graph $G$ is a subset $S$ of $V(G)$ such that every vertex  
in $V(G) \setminus S$ has a neighbor in $S$.
A connected dominating set $S$ of $G$ is a dominating set of $G$ such that 
$\langle S \rangle_G$ is connected.
For a connected graph $G$, 
let $\gamma_{\rm c}(G)$ denote the minimum cardinality of a connected dominating set of $G$.
For $S_1,S_2 \subset V(G)$ such that $S_1 \neq \emptyset$, $S_2 \neq \emptyset$, and
$S_1 \cap S_2 = \emptyset$,
we denote by $\langle S_1,S_2 \rangle_G$ 
the bipartite subgraph of $G$ induced by partite sets $S_1$ and $S_2$,
i.e., $\langle S_1,S_2 \rangle_G = \langle S_1 \cup S_2 \rangle_G - 
(E(\langle S_1 \rangle_G) \cup E(\langle S_2 \rangle_G))$.
A unicyclic graph is a connected graph with exactly one cycle. 
For a nonempty subset $E'$ of $E(G)$, the edge-induced subgraph of $G$ by $E'$ is denoted by 
$\langle E' \rangle_{G}$.
An internal vertex in $G$ is a vertex with degree at least 2 in $G$,
while a leaf in $G$ is a vertex with degree 1 in $G$.
The set of internal vertices (respectively, leaves) in $G$ is denoted by $V_I(G)$
(respectively, $V_L(G)$). 
For a positive integer $k$, let $\mathbb{N}_k = \{1,2,\ldots,k\}$. 

Completely independent spanning trees can be characterized as follows.

\begin{theorem} {\rm \cite{H01}} \label{char-1}
Spanning trees $T_1,T_2,\ldots,T_k$ in a graph $G$ are completely independent if and only if 
$E(T_i) \cap E(T_j) = \emptyset$ and $V_I(T_i) \cap V_I(T_j) = \emptyset$ for any $i \neq j$.
\end{theorem}

Based on Theorem \ref{char-1}, 
a graph which has $k$ completely independent spanning trees can be characterized as follows.

\begin{theorem} \label{char-3}
A graph $G$ has $k$ completely independent spanning trees if and only if
there are $k$ disjoint connected dominating sets $V_1,V_2,\ldots,V_k$ of $G$ 
such that every component of $\langle V_i, V_j \rangle_G$ has a cycle for any $i \neq j$. 
\end{theorem}

\begin{proof}
Suppose that there are $k$ completely independent spanning trees $T_1,T_2,\ldots,T_k$ in $G$. 
Since each $T_i$ is a spanning tree of $G$, each $V_I(T_i)$ is a connected dominating set of $G$.
From Theorem \ref{char-1}, $V_I(T_i) \cap V_I(T_j) = \emptyset$ for any $i \neq j$.
Thus, $V_I(T_1), V_I(T_2), \ldots, V_I(T_k)$ are $k$ disjoint connected dominating sets of $G$.
Let $i, j \in \mathbb{N}_k$ such that $i \neq j$.
Each vertex in $V_I(T_i)$ (respectively, $V_I(T_j)$) is a leaf of $T_j$ (respectively, $T_i$). 
For each $v \in V_I(T_i)$ (respectively, $v \in V_I(T_j)$), 
we denote by $n(v)$ the neighbor of $v$ in $T_j$ (respectively, $T_i$). 
For any $v \in V_I(T_i) \cup V_I(T_j)$, the infinite sequence $(v, n(v), n(n(v)),n(n(n(v))),\ldots)$
of vertices in $\langle V_I(T_i), V_I(T_j) \rangle_{G}$ 
finally becomes periodic but not alternating,
since $E(T_i) \cap E(T_j) = \emptyset$ by Theorem \ref{char-1}. 
This implies that every component of $\langle V_I(T_i), V_I(T_j) \rangle_{G}$ has a cycle.

Suppose that there are $k$ disjoint connected dominating sets $V_1,V_2,\ldots,V_k$ of $G$
such that every component in $\langle V_i, V_j \rangle_G$ has a cycle for any $i \neq j$. 
Every component in $\langle V_i, V_j \rangle_G$ has a spanning unicyclic subgraph
for any $i \neq j$.
Given a unicyclic graph $U$, we may replace the cycle with the directed cycle by orienting
each edge in the cycle and then orient each edge not in the cycle so that every vertex not in 
the cycle can reach a vertex in the cycle through a directed path. 
Thus, any unicyclic graph can be oriented so that every vertex has outdegree one.
Therefore, based on such an orientation, for any $i,j \in \mathbb{N}_k$ 
such that $i \neq j$ and for any $u \in V_i$ (respectively, $v \in V_j$),
we can select one vertex denoted $n'_j(u) \in V_j$ (respectively, $n'_i(v) \in V_i$) so that
$\{ un'_j(u)\ |\ u \in V_i \} \cap \{ vn'_i(v)\ |\ v \in V_j \} = \emptyset$.
Let $R = V(G) \setminus (\cup_{1 \leq i \leq k}V_i)$.
Since each $V_i$ is a dominating set, for any vertex $w$ in $R$,
we can select one vertex $n'_i(w) \in V_i$ adjacent to $w$.
Moreover, since each $\langle V_i \rangle_G$ is connected, it has a spanning tree $H_i$.
Now let $$T'_i = \left\langle E(H_i) \cup 
(\cup_{v \in V(G) \setminus (V_i \cup R)}\{vn'_i(v)\}) \cup 
(\cup_{w \in R}\{wn'_i(w)\}) \right\rangle_G$$ for each $1 \leq i \leq k$.
Then, $T'_1,T'_2,\ldots,T'_k$ are edge-disjoint spanning trees of $G$
such that $V_I(T'_i) \subseteq V_i$ for each $i$ and $R \subseteq V_L(T'_i)$ for all $i$. 
Hence, from Theorem \ref{char-1}, 
$T'_1,T'_2,\ldots,T'_k$ are $k$ completely independent spanning trees in $G$.
$\blacksquare$
\end{proof}

\bigskip

Theorem \ref{char-3} is useful since we can skip exact constructions when proving the existence
of completely independent spanning trees in a graph; however, we may need to use 
Theorem \ref{char-1} when carefully modifying or augmenting completely independent spanning trees. 

A graph $G$ with $k$ completely independent spanning trees can also be characterized 
based on a partition of $V(G)$ \cite{Ar}; 
namely, $G$ has $k$ completely independent spanning trees if and only if 
$G$ has a CIST-partition, i.e., $V(G)$ is partitioned into $V_1,V_2,\ldots,V_k$ such that
$\langle V_i \rangle$ is connected for each $i$ and
$\langle V_i, V_j \rangle_G$ has no tree-component for any $i \neq j$. 
Although a CIST-partition is simpler than the condition in Theorem \ref{char-3},
Theorem \ref{char-3} has an advantage from a fault-tolerant
point of view in communication networks, because any vertices in 
$V(G) \setminus \cup_{1 \leq i \leq k}V_i$ can be deleted while preserving
the existence of $k$ completely independent spanning trees, i.e., 
the following proposition follows from Theorem \ref{char-3}.

\begin{proposition} \label{CIST-fc}
If there are $k$ disjoint connected dominating sets $V_1,V_2,\ldots,V_k$ of a graph $G$ 
such that every component of $\langle V_i, V_j \rangle_G$ has a cycle for any $i \neq j$,
then for any $S \subseteq V(G) \setminus \cup_{1 \leq i \leq k}V_i$, 
there are $k$ completely independent spanning trees in $G-S$.
\end{proposition}

Since completely independent spanning trees are edge-disjoint,
the maximum number of completely independent spanning trees in $G$ is at most
$\min\left\{ \delta(G), \left\lfloor \frac{|E(G)|}{|V(G)|-1} \right\rfloor \right\}$.
From Theorem \ref{char-3}, another general upper bound of 
$\left\lfloor \frac{|V(G)|}{\gamma_{\rm c}(G)} \right\rfloor$ 
on $\tau^\ast(G)$ is obtained.

\begin{proposition} \label{upper} 
Every graph $G$ has 
no $\left\lfloor \frac{|V(G)|}{\gamma_{\rm c}(G)} \right\rfloor + 1$ completely independent spanning trees,
i.e., $\tau^\ast(G) \leq \left\lfloor \frac{|V(G)|}{\gamma_{\rm c}(G)} \right\rfloor$. 
\end{proposition}

The upper bound of 
$\left\lfloor \frac{|V(G)|}{\gamma_{\rm c}(G)} \right\rfloor$ will be used in the next section 
to evaluate the optimality of a lower bound on $\tau^\ast(G)$.

\section{Proof of Theorem \ref{main-0}}

Let $S$ be a star-subset of $V(G)$.
To show the first statement in Theorem \ref{main-0}, 
it is sufficient to show that 
$L(G)$ has $\min\{\tau(G-S), \zeta(S)\}$ completely independent spanning trees.

\bigskip

\noindent Case 1: $S = \emptyset$ or $E(\langle S \rangle_G) = \emptyset$.

Let $T_1,T_2,\ldots,T_{\tau(G-S)}$ be edge-disjoint spanning trees in $G-S$.
Since $\langle E(T_i) \rangle_{L(G)}$ is connected such that 
every edge in $E(G) \setminus (\cup_{1 \leq i \leq \tau(G-S)}E(T_i))$
is adjacent to at least one edge in $E(T_i)$ for each $i$, 
$E(T_1), E(T_2), \ldots, E(T_{\tau(G-S)})$ are disjoint connected dominating sets of $L(G)$.
For any $i \neq j$, every edge in $E(T_i)$ is adjacent to at least two edges in $E(T_j)$
since $V(T_i) = V(T_j)$.
Thus, it holds that $\delta(\langle E(T_i), E(T_j) \rangle_{L(G)}) \geq 2$. 
This means that every component of $\langle E(T_i), E(T_j) \rangle_{L(G)}$ 
has a cycle for any $i \neq j$. 
Therefore, from Theorem \ref{char-3}, it follows that 
$L(G)$ has $\tau(G-S) = \min\{\tau(G-S), \zeta(S)\}$ 
completely independent spanning trees. 
In particular, from the proof of Theorem \ref{char-3}, 
we can construct completely independent spanning trees $T'_1,T'_2,\ldots,T'_{\tau(G-S)}$ in $L(G)$ 
so that every element in $E(G) \setminus (\cup_{1 \leq i \leq \tau(G-S)}E(T_i))$
is a leaf of $T'_i$ for all $1 \leq i \le \tau(G-S)$. 

\bigskip

\noindent Case 2: $S \neq \emptyset$ and $E(\langle S \rangle_G) \neq \emptyset$.

By the argument for Case 1, 
there are $\tau(G-S)$ completely independent spanning trees
$T'_1,T'_2,\ldots,T'_{\tau(G-S)}$ in $L(G) - E(\langle S \rangle_G)$ 
such that every element in $E(G) \setminus 
((\cup_{1 \leq i \leq \tau(G-S)}E(T_i)) \cup E(\langle S \rangle_G))$
is a leaf of $T'_i$ for all $1 \leq i \leq \tau(G-S)$. 
For $w \in S$, let $F_S(w) = \{ ww' \in E(G)\ |\ w' \not\in S\}
= \{ww'_1,ww'_2,\ldots,ww'_{f_S(w)}\}$.
Since $S$ is a star-subset of $V(G)$,
for any edge $xy \in E(\langle S \rangle_G)$, it holds that
${\rm deg}_{\langle S \rangle_G}(x) = 1$ or ${\rm deg}_{\langle S \rangle_G}(y) = 1$.
For each $uv \in E(\langle S \rangle_G)$ where ${\rm deg}_{\langle S \rangle_G}(v) = 1$, 
we augment $T'_i$ by joining $uv$ to $uu'_i$ 
for $1 \leq i \leq \min\{\tau(G-S),f_S(u)\}$ and
to $vv'_{i-\min\{\tau(G-S),f_S(u)\}}$
for $\min\{\tau(G-S),f_S(u)\}+1 \leq i \leq \min\{\tau(G-S),f_S(u)+f_S(v)\}$.
Let $T''_1, T''_2, \ldots, T''_{\tau(G-S)}$ be the resultant trees.
Then, the $\min\{\tau(G-S), \min_{uv \in E(\langle S \rangle_G)}f_S(u)+f_S(v)\} 
= \min\{\tau(G-S), \zeta(S)\}$ trees 
$T''_1, T''_2, \ldots, T''_{\min\{\tau(G-S), \zeta(S)\}}$ are edge-disjoint spanning trees in $L(G)$.
Note that $T''_1, T''_2, \ldots, T''_{\tau(G-S)}$ are indeed edge-disjoint each other but 
$T''_{\min\{\tau(G-S), \zeta(S)\}+1}, \ldots, T''_{\tau(G-S)}$ are not spanning trees in $L(G)$.
Moreover, by our augmentation, it holds that
$V_I(T''_1), V_I(T''_2), \ldots, V_I(T''_{\min\{\tau(G-S), \zeta(S)\}})$ are disjoint each other.
Therefore, from Theorem \ref{char-1}, 
$T''_1, T''_2, \ldots, T''_{\min\{\tau(G-S), \zeta(S)\}}$
are completely independent spanning trees in $L(G)$.

\bigskip

It is known that $K_{n}$ has $\lfloor \frac{n}{2} \rfloor$ edge-disjoint spanning trees.
Since $(\lfloor \frac{n}{2} \rfloor+1)(n-1) > \frac{n(n-1)}{2} = |E(K_n)|$, $K_n$ has no 
$(\lfloor \frac{n}{2} \rfloor + 1)$ edge-disjoint spanning trees.
Thus, $\tau(K_n) = \lfloor \frac{n}{2} \rfloor$.
Let $k$ be a positive integer and $\ell$ a nonnegative integer such that $0 \leq \ell < k$.
Let $H_\ell$ be the graph obtained from $K_{4k}$ by adding two new vertices $u,v$ with 
the edge $uv$ and 
joining $u$ (respectively, $v$) to $\ell$ (respectively, $2k-\ell$) vertices in $K_{4k}$. 
Then, for $S = \{u,v\}$, 
it holds that $\min\{\tau(H_\ell-S),\zeta(S)\} = \min\{\tau(K_{4k}),f_S(u)+f_S(v)\} = 2k$.
Since $S$ is a star-subset of $V(H_\ell)$, $\tau'(H_\ell) \geq 2k$.
On the other hand, $\tau(H_\ell) \leq \delta(H_\ell) = \ell+1 \leq k$.
Hence, it holds that $\tau'(H_\ell) > \tau(H_\ell)$.
Moreover, since $\delta(L(H_\ell)) = 2k$, 
$L(H_\ell)$ has no $\tau'(H_\ell)+1$ completely independent spanning trees.
Therefore, $\tau^\ast(L(H_\ell)) = \tau'(H_\ell)$. 
Hence, the second statement in Theorem \ref{main-0} holds.

\bigskip

By Proposition \ref{upper}, we have 
$\tau^\ast(L(G)) \leq \left\lfloor \frac{|E(G)|}{\gamma_{\rm c}(L(G))} \right\rfloor$.
Let $S$ be a connected dominating set of $L(K_n)$ for $n \geq 3$.
If $|V(\langle S \rangle_{K_n})| \leq n-2$, then $|V(K_n) \setminus V(\langle S \rangle_{K_n})| \geq 2$ 
which implies that $S$ is not a dominating set of $L(K_n)$.
Thus, $|V(\langle S \rangle_{K_n})| \geq n-1$. 
Since $\langle S \rangle_{L(K_n)}$ is connected, $\langle S \rangle_{K_n}$ is also connected.
Therefore, $|S| \geq n-2$.
Thus, $\gamma_{\rm c}(L(K_n)) \geq n-2$.
In fact, the edge set of any subtree of order $n-1$ in $K_n$ is a connected dominating set
of $L(K_n)$.
Hence, it holds that $\gamma_{\rm c}(L(K_n)) = n-2$.
Therefore, we have $\tau^\ast(L(K_n)) \leq \frac{n(n-1)}{2(n-2)} = \frac{n+1}{2}+\frac{1}{n-2}$. 
By the previous arguments, it holds that 
$\tau^\ast(L(K_n)) \geq \tau'(K_n) \geq \tau(K_n) = \lfloor \frac{n}{2} \rfloor$.
Thus, $\lfloor \frac{n}{2} \rfloor \leq \tau^\ast(L(K_n)) \leq \frac{n+1}{2}+\frac{1}{n-2}$.
In particular, for even $n \geq 6$, it follows that $\tau^\ast(L(K_{n})) = \tau(K_{n}) = \frac{n}{2}$.
Hence, the third statement in Theorem \ref{main-0} holds.
$\blacksquare$

\bigskip

Since $L(K_4)$ is 4-regular, it has no three completely independent spanning trees.
Thus, it holds that $\tau^\ast(L(K_{4})) = \tau(K_{4}) = 2$. 
In the next section, we present constructions of $\lfloor \frac{n+1}{2} \rfloor$
disjoint connected dominating sets of $L(K_n)$ satisfying the condition in Theorem \ref{char-3}. 
Thus, it also holds that $\tau^\ast(L(K_{n})) = \frac{n+1}{2}$ for odd $n \geq 5$.
Therefore, $\tau^\ast(L(K_n)) = \lfloor \frac{n+1}{2} \rfloor$ for all $n \geq 4$.

From Theorem \ref{main-0}, we have the following corollaries.

\begin{corollary} 
Let $G$ be a graph and $S \subset V(G)$ such that $E(\langle S \rangle_G) = \emptyset$, 
i.e., $S$ is an independent set.
Then $L(G)$ has $\tau(G-S)$ completely independent spanning trees.
\end{corollary}

\begin{corollary} \label{line-com}
Every line graph $L(G)$ has $\tau(G)$ completely independent spanning trees.
Moreover, there exists a graph $G$ such that $L(G)$ has no $\tau(G)+1$
completely independent spanning trees. 
\end{corollary}

In particular, Corollary \ref{line-com} will be used in Section 5 to prove Theorems \ref{main-1}
and \ref{main-3}.

\section{Proof of Theorem \ref{main-01}}


\begin{figure}[b]
\begin{center}
\epsfig{file=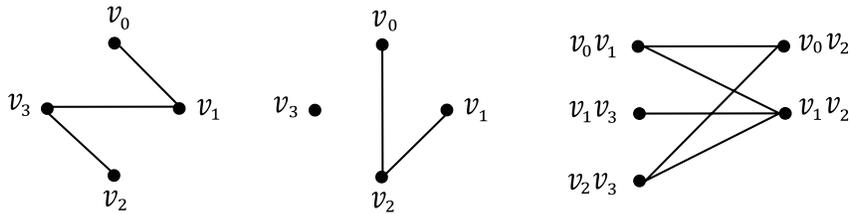,height=33mm}
\caption{Disjoint connected dominating sets $E_1$ and $E_2$ of $L(K_4)$ such that
$\langle E_1, E_2 \rangle_{L(K_4)}$ is a unicyclic graph.}
\end{center}
\end{figure}

Let $n \geq 4$ and $V(K_n) = \{v_0,v_1,\ldots,v_{n-1}\}$.
We first consider the case that $n \leq 6$ and 
define $E_i, 1 \leq i \leq \lfloor \frac{n+1}{2} \rfloor$ as follows:
\begin{itemize}
\item $E_1 = \{ v_0v_1, v_1v_3, v_2v_3 \}$, 
$E_2 = \{ v_0v_2, v_1v_2 \}$ for $n = 4$,
\item $E_1 = \{ v_0v_1, v_1v_4, v_2v_4 \}$, $E_2 = \{v_1v_2, v_0v_2, v_0v_3 \}$, 
$E_3 = \{ v_2v_3, v_3v_4, v_0v_4 \}$ for $n = 5$,
\item 
$E_1 = \{ v_0v_1, v_1v_5, v_2v_5, v_2v_4 \}$, $E_2 = \{v_1v_2, v_0v_2, v_0v_3, v_3v_5 \}$, 
$E_3 = \{ v_2v_3, v_1v_3, v_1v_4, v_0v_4 \}$ for $n = 6$.
\end{itemize}
Then, it can be checked that for each $4 \leq n \leq 6$, 
$E_i, 1 \leq i \leq \lfloor \frac{n+1}{2} \rfloor$ are disjoint connected dominating sets of $L(K_n)$
such that $\langle E_i, E_j \rangle_{L(K_n)}$ is connected and has a cycle for any $i \neq j$
(see Figures 1, 2, and 3).
Note that the edge $v_0v_3$, the edge $v_1v_3$, and the edges $v_3v_4, v_4v_5, v_0v_5$ of $K_n$ 
are not used in the connected dominating sets for $n = 4, 5$, and 6, respectively.
Thus, Theorem \ref{main-01} for $4 \leq n \leq 6$ follow 
from Proposition \ref{CIST-fc} and the symmetry of $K_n$. 
In what follows, we assume that $n \geq 7$.

\begin{figure}[t]
\begin{center}
\epsfig{file=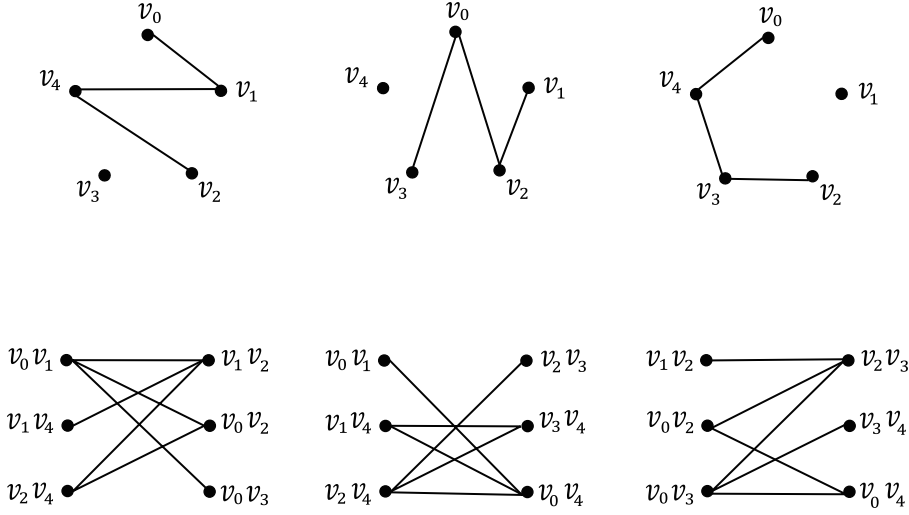,height=78mm}
\caption{Disjoint connected dominating sets $E_1,E_2$, and $E_3$ of $L(K_5)$
such that $\langle E_i, E_j \rangle_{L(K_5)}$ is a unicyclic graph for any $i \neq j$.} 
\end{center}
\end{figure}

\begin{figure}[t]
\begin{center}
\epsfig{file=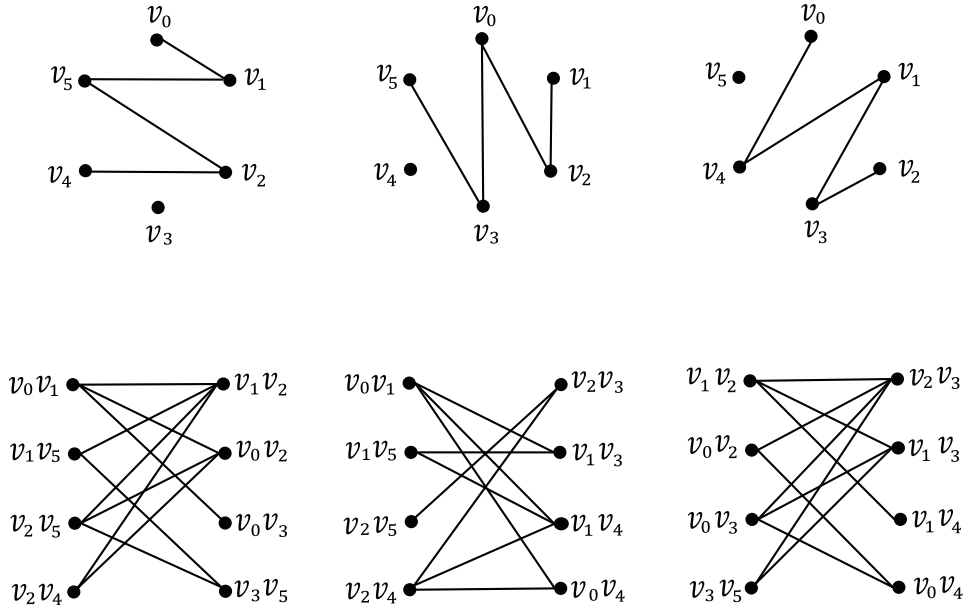,height=90mm}
\caption{Disjoint connected dominating sets $E_1,E_2$, and $E_3$ of $L(K_6)$
such that $\langle E_i, E_j \rangle_{L(K_6)}$ is connected and has a cycle for any $i \neq j$.} 
\end{center}
\end{figure}

Define the subtree $T_i$ of $K_n$ for $0 \leq i < \lfloor \frac{n}{2} \rfloor$
as follows, where the subscripts are expressed modulo $n$: 
$$\left\{ \begin{array}{l}
V(T_i) = V(K_n) \setminus \{ v_{i+\lfloor \frac{n+1}{2} \rfloor} \}, \\
E(T_i) = \left\{v_iv_{i+1},v_{i+1}v_{i-1},v_{i-1}v_{i+2},v_{i+2}v_{i-2},v_{i-2}v_{i+3},\ldots,
v_{i+\lfloor \frac{n-1}{2} \rfloor}v_{i+\lfloor \frac{n+3}{2} \rfloor} \right\}.
\end{array} \right.$$

\begin{figure}[t]
\begin{center}
\epsfig{file=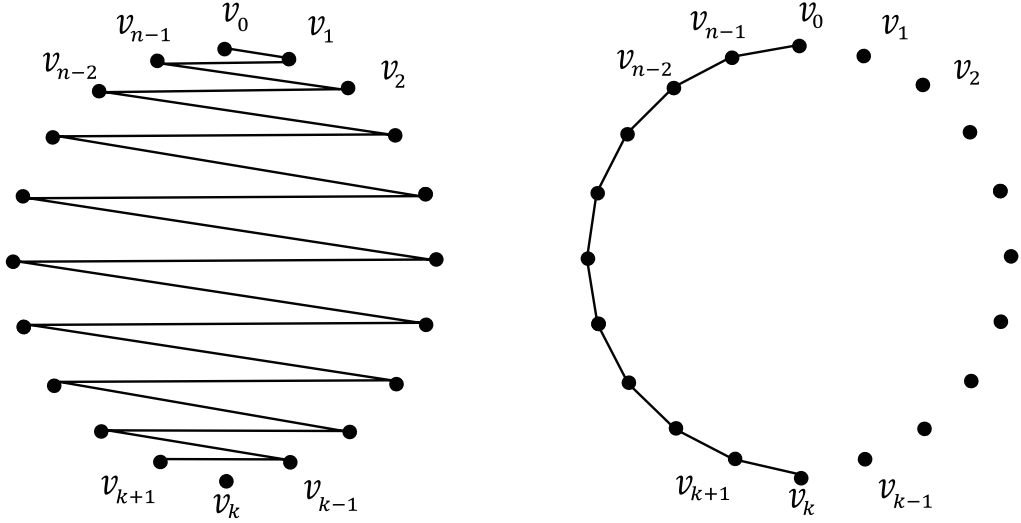,height=82mm}
\caption{The edges of $T_0$ and the edges in 
$E(K_n) \setminus \cup_{0 \leq i < k}E(T_i)$ for even $n = 2k$.} 
\end{center}
\end{figure}

\begin{figure}[t]
\begin{center}
\epsfig{file=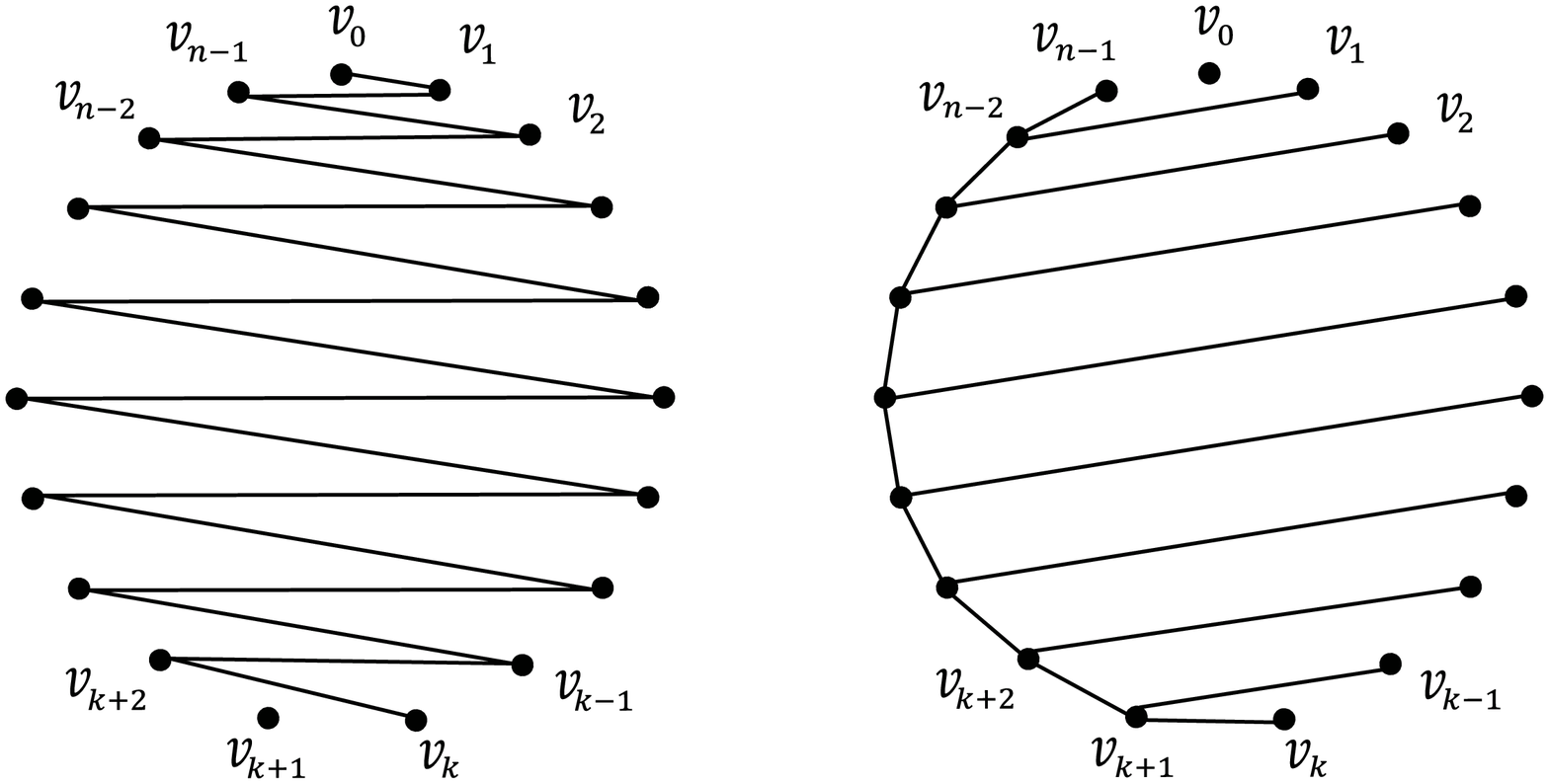,height=82mm}
\caption{The edges of $T_0$ and $T_k$ for odd $n = 2k+1$.} 
\end{center}
\end{figure}

Suppose that the vertices $v_0,v_1,\ldots,v_{n-1}$ are arranged in a regular $n$-gon.
Then, the edges of $T_i$ can be drawn as straight lines in a zig-zag fashion. 
Note that each $T_i$ is isomorphic to the path $P_{n-1}$ of order $n-1$ such that 
each $E(T_i)$ can be obtained from $E(T_0)$ by cyclic shifting along the vertices of the
regular $n$-gon $i$ times.
It can be checked that 
$T_0,T_1,\ldots,T_{\lfloor \frac{n}{2} \rfloor -1}$ are edge-disjoint such that
$$
E(K_n) \setminus \cup_{0 \leq i < \lfloor \frac{n}{2} \rfloor }E(T_i) = \left\{ 
\begin{array}{ll} 
\left\{v_{\frac{n}{2}}v_{\frac{n}{2}+1},
v_{\frac{n}{2}+1}v_{\frac{n}{2}+2},\ldots,v_{n-1}v_0 \right\} & \mbox{if $n$ is even,} \\[2mm]
\left\{v_{\frac{n-1}{2}}v_{\frac{n-1}{2}+1},
v_{\frac{n-1}{2}+1}v_{\frac{n-1}{2}+2},\ldots,v_{n-1}v_0 \right\} & \\
\cup 
\left\{v_{\frac{n-1}{2}+1}v_{\frac{n-1}{2}-1},v_{\frac{n-1}{2}+2}v_{\frac{n-1}{2}-2},\ldots,v_{n-2}v_1 \right\}
& \mbox{if $n$ is odd.}
\end{array} \right.$$
When $n$ is odd, we furthermore define $T_{\lfloor \frac{n}{2} \rfloor}$ as follows: 
$$\left\{ \begin{array}{l}
V(T_{\lfloor \frac{n}{2} \rfloor}) 
= V(K_n) \setminus \{ v_{0} \}, \\
E(T_{\lfloor \frac{n}{2} \rfloor}) 
= (E(K_n) \setminus \cup_{0 \leq i < \lfloor \frac{n}{2} \rfloor}E(T_i)) \setminus \{v_{n-1}v_0 \}.
\end{array} \right.$$
Note that the subtree $T_{\lfloor \frac{n}{2} \rfloor}$ is isomorphic to the graph 
obtained from the path $P_{\frac{n+1}{2}}$ of order $\frac{n+1}{2}$ 
by injectively joining new $\frac{n-3}{2}$ leaves to the internal vertices of $P_{\frac{n+1}{2}}$. 
Thus, every internal vertex of $T_{\lfloor \frac{n}{2} \rfloor}$ has degree 3.
Figures 4 and 5 illustrate the edges of $T_0$ and the edges in 
$E(K_n) \setminus \cup_{0 \leq i < k}E(T_i)$ 
for even $n = 2k$ and the edges of $T_0$ and $T_k$ for odd $n = 2k+1$, respectively.

Since $T_0,T_1,\ldots, T_{\lfloor \frac{n+1}{2}\rfloor-1}$ are edge-disjoint subtrees of order $n-1$
in $K_n$, we can see that 
$E(T_0),E(T_1),\ldots,E(T_{\lfloor \frac{n+1}{2}\rfloor-1})$ are $\lfloor \frac{n+1}{2}\rfloor$ 
disjoint connected dominating sets of $L(K_n)$.
Moreover, 
$\langle E(K_n) \setminus \cup_{0 \leq i < \lfloor \frac{n+1}{2} \rfloor}E(T_i) \rangle_{L(K_n)} 
\cong K_1$
(respectively, $P_{\frac{n}{2}}$) if $n$ is odd (respectively, even).
Thus, from Proposition \ref{CIST-fc} and the symmetry of $K_n$,
in order to show Theorem \ref{main-01},
it is sufficient to show that every component of 
$\langle E(T_i), E(T_j) \rangle_{L(G)}$ has a cycle for any $i \neq j$. 

Let $n = 2k+1$, where $k \geq 3$. 
We only show the case for odd $n \geq 7$, since 
the case for even $n \geq 8$ can be shown similarly to the following argument in Case 1.

\bigskip

\noindent Case 1: $\langle E(T_i), E(T_j) \rangle_{L(G)}$ for $0 \le i < j < k$.

Since $T_i$ is a path with the leaves $v_i$ and $v_{i+k}$, 
${\rm deg}_{T_i}(v_{i}) = {\rm deg}_{T_i}(v_{i+k}) = 1$ and ${\rm deg}_{T_i}(u) = 2$
for any $u \in V(T_i) \setminus \{v_i,v_{i+k}\}$.
Note that $v_{i+k+1} \not\in V(T_i)$ and $v_{i+k}v_{i+k+1} \not\in E(T_j)$.
For any edge $e$ in $E(T_j)$, if $e \neq v_iv_{i+k+1}$, then $e$ is adjacent to 
at least two edges of $T_i$, i.e., ${\rm deg}_{\langle E(T_i),E(T_j) \rangle_{L(G)}}(e) \geq 2$.
Suppose that $v_iv_{i+k+1} \in E(T_j)$.
Then, it holds that ${\rm deg}_{\langle E(T_i),E(T_j) \rangle_{L(G)}}(v_iv_{i+k+1}) = 1$ and
the neighbor of $v_iv_{i+k+1}$ in $\langle E(T_i),E(T_j) \rangle_{L(G)}$ is $v_iv_{i+1}$.
Moreover, it holds that either $v_iv_{i+k}, v_{i+k}v_{i+1} \in E(T_j)$ or 
$v_{i+k+1}v_{i+1}, v_{i+1}v_{i+k} \in E(T_j)$. 
Thus, in either case, it holds that 
${\rm deg}_{\langle E(T_i),E(T_j) \rangle_{L(G)}}(v_iv_{i+1}) \geq 3$.
Therefore, except for at most one edge in $E(T_j)$, every edge in $E(T_j)$ has degree at least 2
in $\langle E(T_i), E(T_j) \rangle_{L(G)}$.
Furthermore, even if $e \in E(T_j)$ has degree 1, the neighbor $n(e) \in E(T_i)$ has degree at least 3
in $\langle E(T_i), E(T_j) \rangle_{L(G)}$.
Such a property also holds for the edges of $E(T_i)$ in $\langle E(T_i), E(T_j) \rangle_{L(G)}$. 
Let $E' = (\{ v_iv_{i+k+1} \} \cap E(T_j)) \cup (\{ v_jv_{j+k+1}\} \cap E(T_i))$.
Since $\delta(\langle E(T_i), E(T_j) \rangle_{L(G)}-E') \geq 2$, 
every component of $\langle E(T_i), E(T_j) \rangle_{L(G)}-E'$ has a cycle. 
The graph $\langle E(T_i), E(T_j) \rangle_{L(G)}$ can be obtained from 
$\langle E(T_i), E(T_j) \rangle_{L(G)}-E'$ by adding leaves if $E' \neq \emptyset$.
Hence, every component of $\langle E(T_i), E(T_j) \rangle_{L(G)}$ has a cycle.

\bigskip

\noindent Case 2: $\langle E(T_i), E(T_k) \rangle_{L(G)}$ for $0 \le i < k$.

Let $e \in E(T_k)$.
If $e \not\in \{ v_iv_{i+k+1}, v_{i+k}v_{i+k+1} \}$, then  
${\rm deg}_{\langle E(T_i), E(T_k) \rangle_{L(G)}}(e) \geq 2$.  
In $\langle E(T_i), E(T_k) \rangle_{L(G)}$, $v_{i+k}v_{i+k+1} \in E(T_k)$ has degree 1, while 
its neighbor $v_{i+k}v_{i+k+2} \in E(T_i)$ has degree at least 3,
since ${\rm deg}_{T_k}(v_{i+k}) = 3$ or ${\rm deg}_{T_k}(v_{i+k+2}) = 3$.
Note that $v_{i+k+2} \not\in V(T_k)$ if $i = k-1$. 
Suppose that $v_iv_{i+k+1} \in E(T_k)$.
Then, $k$ is odd and $i = \frac{k-1}{2} \geq 1$.
Moreover, ${\rm deg}_{\langle E(T_i), E(T_k) \rangle_{L(G)}}(v_iv_{i+k+1}) = 1$
and the neighbor $v_iv_{i+1}$ of $v_iv_{i+k+1}$ has degree 2 in $\langle E(T_i), E(T_k) \rangle_{L(G)}$
such that $v_iv_{i+1}$ is adjacent to $v_{i+1}v_{i+k}$ in $\langle E(T_i), E(T_k) \rangle_{L(G)}$.
Since ${\rm deg}_{T_i}(v_{i+1}) = 2$ and ${\rm deg}_{T_i}(v_{i+k}) = 1$, 
it holds that ${\rm deg}_{\langle E(T_i), E(T_k) \rangle_{L(G)}}(v_{i+1}v_{i+k}) = 3$.
Therefore, the leaf $v_iv_{i+k+1}$ is connected to $v_{i+1}v_{i+k}$ with degree 3 through 
the induced subpath 
$\langle \{v_iv_{i+k+1}, v_iv_{i+1}, v_{i+1}v_{i+k}\} \rangle_{\langle E(T_i), E(T_k) \rangle_{L(G)}}$.

Let $e' \in E(T_i)$.
If $e' \not\in \{v_0v_1, v_0v_2, \ldots, v_0v_k\}$, then 
${\rm deg}_{\langle E(T_i), E(T_k) \rangle_{L(G)}}(e') \geq 2$.  
Suppose that $v_0v_p \in E(T_i)$ where $1 \leq p \leq k$. 
Note that $i = \lfloor \frac{p}{2} \rfloor$.
Since $v_0 \not\in V(T_k)$, ${\rm deg}_{\langle E(T_i), E(T_k) \rangle_{L(G)}}(v_0v_p) = 1$ and 
the neighbor of $v_0v_p$ in $\langle E(T_i), E(T_k) \rangle_{L(G)}$ is 
$v_pv_{n-1-p} \in E(T_k)$ (respectively, $v_pv_{p+1}$) if $p < k$ (respectively, $p = k$). 
If $p = k$, then $v_{p+1} \neq v_{i+k+1}$ since $i = \lfloor \frac{p}{2} \rfloor \geq 1$.
Thus, in such a case, ${\rm deg}_{\langle E(T_i), E(T_k) \rangle_{L(G)}}(v_pv_{p+1}) \geq 3$.
Suppose that $p < k$.  
If $v_{n-1-p} \neq v_{i+k+1}$, then $v_pv_{n-1-p}$ has degree at least 3
in $\langle E(T_i), E(T_k) \rangle_{L(G)}$.
Suppose that $v_{n-1-p} = v_{i+k+1}$.
Note that $p \geq 2$.
Since $v_{i+k+1} \not\in V(T_i)$, 
it holds that 
${\rm deg}_{\langle E(T_i), E(T_k) \rangle_{L(G)}}(v_pv_{n-1-p}) = 2$ and
the neighbor $n(v_pv_{n-1-p})$ of $v_pv_{n-1-p}$ different from $v_0v_p$ is 
$v_1v_p$ (respectively, $v_{n-1}v_p$) if $p$ is even (respectively, odd). 
Thus, $n(v_pv_{n-1-p})$ has degree 2 
in $\langle E(T_i), E(T_k) \rangle_{L(G)}$ and the neighbor $n(n(v_pv_{n-1-p}))$ 
of $n(v_pv_{n-1-p})$ different from $v_pv_{n-1-p}$ is $v_1v_{n-2}$ (respectively, $v_{n-1}v_{n-2}$)
if $p$ is even (respectively, odd).
Thus, $n(n(v_pv_{n-1-p}))$ is incident to $v_{n-2}$ in $T_k$.
Note that $v_{i+k+1} = v_{n-1-p} < v_{n-2}$. 
Therefore, $v_{n-2}$ has degree 2 in $T_i$ which means that 
$n(n(v_pv_{n-1-p}))$ has degree at least 3
in $\langle E(T_i), E(T_k) \rangle_{L(G)}$.
Hence, the leaf $v_0v_p$ is connected to $n(n(v_pv_{n-1-p}))$ with degree at least 3 through 
the induced subpath 
$\langle \{v_0v_p, v_pv_{n-1-p}, n(v_pv_{n-1-p}), n(n(v_pv_{n-1-p}) \} \rangle_{\langle E(T_i), E(T_k) \rangle_{L(G)}}$. 

Consequently, 
for any leaf in $\langle E(T_i), E(T_k) \rangle_{L(G)}$, 
it is adjacent or connected (through an induced path of order at most 4) 
to a vertex with degree at least 3.
In particular, $E(T_k)$ has at most two leaves in $\{v_iv_{i+k+1}, v_{i+k}v_{i+k+1}\}$ 
and at most three vertices in $\{v_{i+1}v_{i+k}, v_{2i}v_{n-1-2i}, v_{2i+1}v_{n-2-2i} \} \cup
\{v_1v_{n-2}, v_{n-2}v_{n-1}, v_kv_{k+1} \}$
with degree at least 3 adjacent or connected to a leaf in $E(T_i) \cup E(T_k)$.
Note that $E(T_k)$ indeed has the leaf $v_{i+k}v_{i+k+1}$.
On the other hand, $E(T_i)$ has at most two leaves in $\{v_0v_{2i}, v_0v_{2i+1}\}$ 
and the vertex $v_{i+k}v_{i+k+2}$ with degree at least 3 adjacent to the leaf 
$v_{i+k}v_{i+k+1} \in E(T_k)$.
Suppose that $v_iv_{i+k+1}$ is a leaf in $\langle E(T_i), E(T_k) \rangle_{L(G)}$. 
Since $i = \frac{k-1}{2} \geq 1$, it holds that $v_{k+1} \neq v_{i+k+1}$
and $E(T_i)$ has two leaves $v_0v_{k-1}$ and $v_0v_k$ adjacent to $v_{k-1}v_{k+1}$ and $v_kv_{k+1}$,
respectively, such that both $v_{k-1}v_{k+1}$ and $v_kv_{k+1}$ have degree at least 3.
If $v_iv_{i+k+1}$ is not a leaf in $\langle E(T_i), E(T_k) \rangle_{L(G)}$, i.e.,
$v_iv_{i+k+1} \not\in E(T_k)$, then 
there are at most two leaves adjacent or connected to a vertex in $E(T_k)$ 
with degree at least 3. 
Therefore, in any case, it does not happen that three leaves are adjacent or connected to 
the same vertex in $E(T_k)$ with degree at least 3. 
Thus, every leaf in $\langle E(T_i), E(T_k) \rangle_{L(G)}$
can be deleted (with an induced path of order at most 3) 
so that the resulting bipartite subgraph has at most one leaf.
Hence, every component in $\langle E(T_i), E(T_k) \rangle_{L(G)}$ has a cycle.
$\blacksquare$

\section{Proofs of Theorems \ref{main-1} and \ref{main-3} }

Edge-connectivity of a graph was characterized by Catlin (see \cite{CLS,OY}) in terms of 
edge-disjoint spanning trees as follows.

\begin{theorem} \label{char-edge-con}
Let $G$ be a connected graph and let $k$ be a positive integer.
Then, $\lambda(G) \geq k$ if and only if for any $X \subseteq E(G)$
with $|X| \leq \lceil \frac{k}{2} \rceil$, $\tau(G-X) \geq \lfloor \frac{k}{2} \rfloor$.
\end{theorem}

A $p$-$q$-{\em edge-cut} of a connected graph $G$ is an edge-cut $F$ of $G$ 
such that one component of $G-F$
has at least $p$ vertices and another component of $G-F$ has at least $q$ vertices.
The $p$-$q$-{\em restricted edge-connectivity} $\lambda_{p,q}(G)$ of $G$ is defined to be 
the minimum cardinality of a $p$-$q$-edge-cut of $G$ if $G$ has a $p$-$q$-edge-cut.  
It has been shown in \cite{EH} that 
except for a star, any graph of order at least 4 has a $2$-$2$-edge-cut. 

Hellwig et al. proved that the connectivity of $L(G)$ is the same as 
the 2-2-restricted edge-connectivity of $G$.
In particular, it follows from Theorem \ref{con-L} that $\kappa(L(K_n)) = 2n-4$ for all $n \geq 4$. 

\begin{theorem} \label{con-L} {\rm \cite{HRV}}
Let $G$ be a connected graph of order $n \geq 4$ such that $G$ is not a star.
Then, $\kappa(L(G)) = \lambda_{2,2}(G)$.
\end{theorem}

Combining Corollary \ref{line-com} and Theorems \ref{char-edge-con} and \ref{con-L},
we have the following.

\begin{theorem} \label{S5-main1}
Let $G$ be a connected graph.
For any proper subset $S \subset V(L(G))$ 
with $|S| \leq \left\lceil \frac{\min\{\delta(G),\kappa(L(G))\}}{2} \right\rceil$, 
$L(G)-S$ has $\left\lfloor \frac{\min\{\delta(G),\kappa(L(G))\}}{2} \right\rfloor$ 
completely independent spanning trees.
Moreover, if $G$ is not super edge-connected, then 
for any $S \subset V(L(G))$ with $|S| \leq \left\lceil \frac{\kappa(L(G))}{2} \right\rceil$, 
$L(G)-S$ has $\left\lfloor \frac{\kappa(L(G))}{2} \right\rfloor$ 
completely independent spanning trees.
\end{theorem}

\begin{proof}
If $\left\lfloor \frac{\min\{\delta(G),\kappa(L(G))\}}{2} \right\rfloor = 0$, then the theorem vacuously holds.
Suppose that $\left\lfloor \frac{\min\{\delta(G),\kappa(L(G))\}}{2} \right\rfloor = 1$.
Then, $1 \leq \left\lceil \frac{\min\{\delta(G),\kappa(L(G))\}}{2} \right\rceil \leq 2$ and 
$\kappa(L(G)) \geq 2$.
If $\kappa(L(G)) \geq 3$, then $L(G)-S$ is connected. 
If $\kappa(L(G)) = 2$, then
$\left\lceil \frac{\min\{\delta(G),\kappa(L(G))\}}{2} \right\rceil = 1$ and 
$L(G)-S$ is also connected. 
Thus, the theorem holds since $L(G)-S$ has a spanning tree.  

Suppose that $\left\lfloor \frac{\min\{\delta(G),\kappa(L(G))\}}{2} \right\rfloor \geq 2$.
Then, $\delta(G) \geq 4$. 
Thus, the order of $G$ is at least 5 and $G$ is not a star.
Therefore, from Theorem \ref{con-L}, it follows that $\kappa(L(G)) = \lambda_{2,2}(G)$.
By definition, $\delta(G) \geq \lambda(G)$ and $\lambda_{2,2}(G) \geq \lambda(G)$, i.e.,
$\min\{ \delta(G), \lambda_{2,2}(G) \} \geq \lambda(G)$.
If $G$ is super edge-connected, then $\lambda(G) = \delta(G)$.
If $G$ is not super edge-connected, then $\lambda(G) = \lambda_{2,2}(G)$.
Thus, 
$\lambda(G) = \min\{ \delta(G), \lambda_{2,2}(G) \} = \min\{ \delta(G), \kappa(L(G)) \}$.
Hence, the statements in Theorem \ref{S5-main1} 
follows from Corollary \ref{line-com} and Theorem \ref{char-edge-con}.
Note that it holds that $L(G-S) = L(G)-S$. 
$\blacksquare$
\end{proof}

\bigskip

Applying Theorem \ref{S5-main1} to $2k$-connected line graphs, 
we have the following from which Theorem \ref{main-1} is obtained by setting
$S = \emptyset$.

\begin{theorem} \label{ext}
For any $2k$-connected line graph $L(G)$ and any $S \subset V(L(G))$ with $|S| \leq k$,
$L(G)-S$ has $k$ completely independent spanning trees 
if $G$ is not super edge-connected or $\delta(G) \geq 2k$. 
\end{theorem}

The restrictions in Theorem \ref{main-1} can be weakened as follows.

\begin{theorem} \label{weak-con1}
For any $k \geq 2$, 
every $2k$-connected line graph $L(G)$ has $k$ completely independent 
spanning trees if $G$ is a star or there exists a graph $G^\ast \supseteq G$ with 
$|E(G^\ast) \setminus E(G)| \leq k$ 
such that $G^\ast$ is not super edge-connected or $\delta(G^\ast) \geq 2k$. 
\end{theorem}

\begin{proof}
We may assume that $|V(G)| \geq 5$.
Suppose that $G$ is a star and $L(G)$ is $2k$-connected. 
Then $L(G) \cong K_n$ where $n \geq 2k+1$.
Let $V_1,V_2,\ldots,V_k \subset V(K_n)$
such that $V_i \cap V_j = \emptyset$ for any $i \neq j$ and 
$|V_i| = 2$ for all $1 \leq i \leq k$.
Clearly, $V_1,V_2,\ldots,V_k$ are $k$ disjoint connected dominating sets
of $K_n$ such that $\langle V_i,V_j \rangle_{K_n}$ is a cycle of order 4 for any $i \neq j$. 
Thus, from Theorem \ref{char-3}, $L(G)$ has
$k$ completely independent spanning trees. 

Suppose that $G$ is not a star, $L(G)$ is $2k$-connected, and 
there exists a graph $G^\ast \supseteq G$ with the conditions.
By Theorem \ref{con-L}, $\lambda_{2,2}(G) = \kappa(L(G))$.
Since $G^\ast$ is obtained from $G$ by adding new edges,
it holds that $\kappa(L(G^\ast)) = \lambda_{2,2}(G^\ast) \geq \lambda_{2,2}(G) \geq 2k$.
Hence, from Theorem \ref{ext}, 
$L(G^\ast)-(E(G^\ast) \setminus E(G)) = L(G)$ has $k$ completely independent spanning trees.
$\blacksquare$
\end{proof}

\bigskip

An edge-cut $X$ of a connected graph $G$ is {\em essential} if at least two components of $G-X$
are nontrivial. 
A graph is {\em essentially $k$-edge-connected} if it has no essential edge-cut
with fewer than $k$ edges \cite{LL}.
Note that an essential edge-cut is the same as a 2-2-edge-cut.
Lai and Li proved the following.
They also showed that both the lower bounds on $g$ and $h$ are tight.

\begin{theorem} {\rm \cite{LL} } \label{LL}
Let $k$, $g$, $h$ be positive integers such that $k < g < 2k$ and
$h \geq \frac{g^2}{g-k}-2$.
Then, every $g$-edge-connected and essentially $h$-edge-connected graph contains 
$k$ edge-disjoint spanning trees.
\end{theorem}

Combining Corollary \ref{line-com} and Theorems \ref{con-L} and \ref{LL}, we have the following.
Note that the condition on $G$ in Theorem \ref{main-2}
is complementary to that in Theorem \ref{ext}. 

\begin{theorem} \label{main-2}
For any $k \geq 2$, 
every $\left( \left\lceil \frac{\delta(G)^2}{\delta(G)-k} \right\rceil -2 \right)$-connected line graph
$L(G)$ has $k$ completely independent spanning trees if $G$ is super edge-connected 
and $k < \delta(G) < 2k$.
\end{theorem}

\begin{proof}
Suppose that $G$ is super edge-connected, $k < \delta(G) < 2k$ where $k \geq 2$, 
and $L(G)$ is $\left(\left\lceil \frac{\delta(G)^2}{\delta(G)-k} \right\rceil-2 \right)$-connected.
Since $G$ is super edge-connected, it holds that $\lambda(G) = \delta(G)$.
Since $\delta(G) > k \geq 2$, $|V(G)| \geq 4$ and $G$ is not a star.
Thus, by Theorem \ref{con-L}, 
$\lambda_{2,2}(G) = \kappa(L(G)) \geq \left\lceil \frac{\delta(G)^2}{\delta(G)-k} \right\rceil-2$.
Hence, $G$ is $\delta(G)$-edge-connected and essentially 
$\left( \left\lceil \frac{\delta(G)^2}{\delta(G)-k} \right\rceil-2 \right)$-edge-connected.
Therefore, by Corollary \ref{line-com} and Theorem \ref{LL}, $L(G)$ has $k$ completely 
independent spanning trees.
$\blacksquare$
\end{proof}

\bigskip

Now let $L(G)$ be $(k^2+2k-1)$-connected and $\delta(G) \geq k+1$, where $k \geq 2$. 
If $G$ is not super edge-connected or $\delta(G) \geq 2k$,
then from Theorem \ref{ext}, $L(G)$ has $k$ completely independent spanning trees.
Suppose that $G$ is super edge-connected and $k < \delta(G) < 2k$. 
It can be checked that 
the function $f(x) = \frac{x^2}{x-k}$ with domain $\mathbb{R} \setminus \{ k \}$ 
is monotonically decreasing in the interval $k < x < 2k$.
Thus, it holds that $(k+1)^2 \geq \frac{\delta(G)^2}{\delta(G)-k} > 4k$ for $k < \delta(G) < 2k$,
i.e., $k^2+2k-1 \geq \left\lceil \frac{\delta(G)^2}{\delta(G)-k} \right\rceil -2$
for $k < \delta(G) < 2k$.
Therefore, from Theorem \ref{main-2}, 
$L(G)$ also has $k$ completely independent spanning trees.
Hence, Theorem \ref{main-3} holds.

\bigskip

\section*{Acknowledgements}

This work was supported by JSPS KAKENHI Grant Number JP19K11829.

\end{document}